\documentclass[a4paper]{amsart}%
\usepackage{amscd}
\usepackage{amsthm}
\usepackage{graphicx}
\usepackage{amsmath}
\usepackage{amsfonts}
\usepackage{amssymb}%
\providecommand{\U}[1]{\protect\rule{.1in}{.1in}}
\newtheorem{theorem}{Theorem}

\newtheorem{definition}[theorem]{Definition}

\newtheorem{lemma}[theorem]{Lemma}

\newtheorem{proposition}[theorem]{Proposition}
\newtheorem{remark}[theorem]{Remark}

\begin{document}

\title[Generalized Riesz Potentials on the Variable Hardy Spaces]{Boundedness of Generalized Riesz Potentials on the Variable Hardy Spaces}
\author{Pablo Rocha}
\address{Universidad Nacional del Sur, INMABB (Conicet), Bah\'{\i}a Blanca, 8000 Buenos
Aires, Argentina.}
\email{pablo.rocha@uns.edu.ar}
\thanks{\textbf{2.010 Math. Subject Classification}: 42B25, 42B35.}
\thanks{Partially supported by SecytUNC and UNS}
\maketitle

\begin{abstract}
We study the boundedness from $H^{p(.)}(\mathbb{R}^{n})$ into $L^{q(.)}(\mathbb{R}^{n})$ of certain generalized Riesz potentials and the boundedness from $H^{p(.)}(\mathbb{R}^{n})$ into $H^{q(.)}(\mathbb{R}^{n})$ of the Riesz potential, both results are achieved via the finite atomic decomposition developed in \cite{Uribe}.
\end{abstract}

\section{Introduction}

The Lebesgue spaces with variable exponents are a generalized of the classical Lebesgue spaces, replacing the constant exponent $p$ with a variable exponent function $p(\cdot)$. In the last 20 years, the variable Lebesgue spaces have received considerable attention for their applications in fluid dinamics, elasticity
theory and differential equations with non-standard growth conditions, see \cite{diening2} and references therein.

In the celebrated paper \cite{F-S}, C. Fefferman and E. Stein defined the Hardy spaces $H^{p}(\mathbb{R}^{n})$ for $0 < p < \infty$. One of the principal interests of $H^{p}$ theory is that the $L^{p}$ boundedness of certain integrals operators proved for $p >1$ extend to the context of $H^{p}$, for all $p > 0$.
In many cases, this is achieved by mean of the atomic decomposition of the elements in $H^{p}$ (see \cite{garcia}, \cite{Lu}, \cite{steinlibro}).

The theory of variable exponent Hardy spaces was developed independently by E. Nakai and Y. Sawano in \cite{nakai} and by D. Cruz-Uribe and D. Wang in \cite{Uribe}. Both theories prove equivalent definitions in terms of maximal operators using different approach. In \cite{nakai} and \cite{Uribe}, one of their main goals is the atomic decomposition of the elements in $H^{p(\cdot)}(\mathbb{R}^{n})$, as an application of such decomposition they proved the boundedness the certain singular integrals on  $H^{p(\cdot)}(\mathbb{R}^{n})$.

Let $0\leq \alpha <n$ and $m > 1$, ($m \in \mathbb{N}$), we consider the following generalization of Riesz potential
\begin{equation}
T_{\alpha, m}f(x)=\int_{\mathbb{R}^{n}} \left\vert x-A_{1}y\right\vert ^{-\alpha
_{1}}...\left\vert x-A_{m}y\right\vert ^{-\alpha _{m}}f(y)dy,  \label{T}
\end{equation}%
where $\alpha _{1}+...+\alpha _{m}=n-\alpha$, and $A_1, ..., A_m$ are $n \times n$ orthogonal matrices such that $A_i - A_j$ are invertible if $i \neq j$. We observe that in the case $%
\alpha >0,$ $m=1$ and $A_{1}=I,$ $T_{\alpha, 1}$ is the classical fractional integral
operator (also known as the Riesz potential) $I_{\alpha }.$ A interesting survey about fractional integrals can be founded in \cite{nakai0}.

With respect to classical Lebesgue or Hardy spaces, for the case $0 \leq \alpha < n$ and  $m>1,$ in
the paper \cite{rocha1}, the author jointly with M. Urciuolo proved the $H^{p}(\mathbb{R}^{n})-L^{q}(\mathbb{R}
^{n})$ boundedness of the operator $T_{\alpha, m}$ and we also showed that the $H^{p}(\mathbb{R})-H^{q}(\mathbb{R})$ boundedness cannot expect for $T_{\alpha, m}$ with $0 \leq \alpha < 1$, $m=2$, $A_1 = 1$, and $A_2=-1$.
This is an important difference with the case $0 < \alpha < n$ and $m=1.$ Indeed, in the paper \cite{T-W},
M. Taibleson and G. Weiss, using the molecular characterization of the real
Hardy spaces, obtained the boundedness of $I_{\alpha }$ from $H^{p}(\mathbb{R%
}^{n})$ into $H^{q}(\mathbb{R}^{n}),$ for $0<p\leq 1$ and $\frac{1}{q}=\frac{1}{p} - \frac{\alpha}{n}$. 

In \cite{capone}, it gives the boundedness of the Riesz potential from $L^{p(\cdot)}$ into  $L^{q(\cdot)}$ where $\frac{1}{q(\cdot)} = \frac{1}{p(\cdot)} - \frac{\alpha}{n}$. Y. Sawano obtains in \cite{sawano} the $H^{p(\cdot)}-H^{q(\cdot)}$ boundedness of the Riesz potential using a finite atomic decomposition different of given in \cite{Uribe}. In \cite{rocha2}, the author jointly with M. Urciuolo proved the $H^{p(\cdot)}-L^{q(\cdot)}$ boundedness of the operator $T_{\alpha, m}$ and the $H^{p(\cdot)}-H^{q(\cdot)}$ boundedness of the Riesz potential via the infinite atomic and molecular decomposition developed in \cite{nakai}.

The purpose of this article, it is give other proof of the results obtained in \cite{rocha2}, but now using the finite atomic decomposition developed in \cite{Uribe}. Here, a key tool is a weighted vector-valued inequality for the fractional maximal operator. We also rely on the theory of weighted Hardy spaces and on the Rubio de Francia iteration algorithm.

This method allows us to avoid the more delicate convergence arguments that are often neccesary when utilizing the infinite atomic decomposition.

\qquad

In Section 2 we give some basics results  about the variable Lebesgue spaces and the theory of weights. We also recall the definition and atomic decomposition of the variable Hardy spaces given in \cite{Uribe}. In Section 3 we state some auxiliary lemmas and propositions to get the main results of Section 4.

\qquad

\textbf{Notation:} We denote by $B(x_0, r)$ the ball centered at $x_0 \in \mathbb{R}^{n}$ of radius $r$. For a measurable subset $E \subset \mathbb{R}^{n}$ we denote $|E|$ and $\chi_E$ the Lebesgue measure of $E$ and the characteristic function of $E$ respectively. Given a real number $s \geq 0$,
we denote $\lfloor s \rfloor$ as the smallest integer $k$ with $k > s-1$. As usual we denote with $\mathcal{S}(\mathbb{R}^{n})$ the space os smooth and rapidly decreasing functions, with $\mathcal{S}'(\mathbb{R}^{n})$  the dual space. If $\beta$ is the multiindex $\beta=(\beta_1, ..., \beta_n)$, then
$|\beta| = \beta_1 + ... + \beta_n$.

Throughout this paper, $C$ will denote a positive constant, not necessarily the same at each occurrence.

\section{Preliminaries}

In this section, we give some definitions and some basics results about the variable Lebesgue spaces and the theory of weights.

Given a measurable function $p(\cdot) : \mathbb{R}^{n} \rightarrow (0, \infty)$ such that $$0 < ess \inf_{x \in \mathbb{R}^{n}} p(x) \leq ess \sup_{x \in \mathbb{R}^{n}} p(x) < \infty,$$
let $L^{p(\cdot)}(\mathbb{R}^{n})$ denote the space of all measurable functions $f$ such that for some $\lambda > 0$,
$$ \int_{\mathbb{R}^{n}} \, \left| \frac{f(x)}{\lambda} \right|^{p(x)} \, dx < \infty.$$
We set $$\| f \|_{p(\cdot)} = \inf \left\{ \lambda > 0 : \int_{\mathbb{R}^{n}} \, \left| \frac{f(x)}{\lambda} \right|^{p(x)} \, dx \leq 1 \right\}.$$
We see that $\left( L^{p(\cdot)}(\mathbb{R}^{n}), \| . \|_{p(\cdot)} \right)$ is a quasi normed space.

Here we adopt the standard notation in variable exponents. We write
$$p_{-}= ess \inf_{x \in \mathbb{R}^{n}} p(x), \,\,\,\,\,\, p_{+} = ess \sup_{x \in \mathbb{R}^{n}} p(x), \,\,\,\, \text{and} \,\,\,\, \underline{p}=\min \left\{ p_{-},1\right\}.$$
From now on we assume $0 < p_{-} \leq p_{+} < \infty$. It not so hard to see the following

\begin{lemma} The following statements hold 

$1.$ $\| f \|_{p(\cdot)} \geq 0$, and $\| f \|_{p(\cdot)}=0$ if and only if $f \equiv 0$.

$2.$ $\| c \, f \|_{p(\cdot)} = |c| \, \| f \|_{p(\cdot)}$ for $c \in \mathbb{C}$.

$3.$ $\| f + g \|_{p(\cdot)}^{\underline{p}} \leq \| f \|_{p(\cdot)}^{\underline{p}} + \| g \|_{p(\cdot)}^{\underline{p}}$.

$4.$ $\| |f|^{s} \|_{p(\cdot)} = \| f \|^{s}_{sp(\cdot)}$, for all $s > 0$.

$5.$ If $A$ is an $n \times n$ orthogonal matrix and $p(Ax)=p(x)$ for all $x \in \mathbb{R}^{n}$, then
$\| f_{A} \|_{p(.)} = \| f \|_{p(.)},$ where $f_{A}(x) = f(A^{-1}x)$.
\end{lemma}

Let $\mathcal{P}_0$ denote the collection of all measurable functions $p(.):\mathbb{R}^{n}\rightarrow \left( 0,\infty \right) $ such that $0 < p_{-} \leq p_{+} < \infty.$

Given $p(\cdot) \in \mathcal{P}_0$ with $p_{-} > 1$, define the conjugate exponent $p'(\cdot)$ by the equation $\frac{1}{p(\cdot)} + \frac{1}{p'(\cdot)}=1$.
\begin{lemma} $($See Theorem 2.34 in \cite{uribe}$)$ If $p(\cdot) \in \mathcal{P}_0$ with $p_{-} > 1$, then for all $f \in L^{p(\cdot)}$, we have
\[
\| f \|_{p(\cdot)} \leq C \sup \int_{\mathbb{R}^{n}} |f(x) g(x)| dx,
\]
where the supremum is taken over all $g \in L^{p'(\cdot)}$ such that $\| g \|_{p'(\cdot)} \leq 1$.
\end{lemma}

Let $f$ be a locally integrable function on $\mathbb{R}^{n}$. The function
$$M(f)(x) = \sup_{B \ni x} \frac{1}{|B|} \int_{B} |f(y)| dy,$$
where the supremum is taken over all balls $B$ containing $x$, is called the uncentered Hardy-Littlewood maximal function of $f$.

Throughout, we will make use of the following class of exponents.
\begin{definition} Given $p(\cdot) \in \mathcal{P}_{0}$, we say $p(\cdot) \in M\mathcal{P}_0$ if there exists $p_0$, $0< p_0 < p_{-}$, such that $\| Mf \|_{\frac{p(\cdot)}{p_0}} \leq C \| f \|_{\frac{p(\cdot)}{p_0}}.$
\end{definition}

\begin{lemma} Given $p(\cdot) \in \mathcal{P}_0$, with $p_{-} >1$, if the maximal operator is bounded on $L^{p(\cdot)}$, then for every $s>1$ it is bounded on $L^{sp(\cdot)}$.
\end{lemma}
\begin{proof} From H\"older's inequality and Lemma 1, 4., we have $$\|Mf\|_{sp(\cdot)} = \| (Mf)^{s} \|_{p(\cdot)}^{1/s} \leq \| (M(|f|^{s}) \|_{p(\cdot)}^{1/s}
\leq C \| |f|^{s} \|_{p(\cdot)}^{1/s} = C \| f \|_{sp(\cdot)}.$$
\end{proof}

The following lemma is a deep result due to L. Diening (see Theorem 8.1 in \cite{diening}).

\begin{lemma} Given $p(\cdot) \in \mathcal{P}_0$, with $p_{-} >1$, the maximal operators is bounded on $L^{p(\cdot)}$ if and only if it is bounded on $L^{p'(\cdot)}�$.
\end{lemma}

It is well known that a useful sufficient condition for the boundedness of the maximal operator is log-H\"older continuity, (see \cite{uribe}, \cite{diening2}). In our main results we only will assume that the exponents $p(\cdot)$ belongs to $M\mathcal{P}_0$.

\qquad

In the paper \cite{Uribe}, D. Cruz-Uribe and D. Wang give a variety of distinct
approaches, based on differing definitions, all lead to the same notion of the
variable Hardy space $H^{p(.)}.$

We recall the definition and the atomic decomposition of the Hardy spaces with variable exponents.

Topologize $\mathcal{S}(\mathbb{R}^{n})$ by the collection of semi-norms $\| \cdot \|_{\alpha, \beta}$, with $\alpha$ and $\beta$ multi-indices, given by
$$\| \varphi \|_{\alpha, \beta} = \sup_{x \in \mathbb{R}^{n}} |x^{\alpha} \partial^{\beta}\varphi(x)|.$$
For each $N \in \mathbb{N}$, we set $\mathcal{S}_{N}=\left\{  \varphi\in \mathcal{S}(\mathbb{R}^{n}): \| \varphi \|_{\alpha, \beta} \leq 1, |\alpha|, |\beta| \leq N \right\}$. Let $f \in \mathcal{S}'(\mathbb{R}^{n})$, we denote by $\mathcal{M}_{N}$ the grand maximal
operator given by
\[
\mathcal{M}_{N}f(x)=\sup\limits_{t>0}\sup\limits_{\varphi\in\mathcal{S}_{N}
}\left\vert \left(  t^{-n}\varphi(t^{-1} \cdot)\ast f\right)  \left(  x\right)
\right\vert.
\]

\begin{definition} Let $p(\cdot) \in M\mathcal{P}_0$. For $N > \frac{n}{p_0} + n + 1$, define
the Hardy space with variable exponents $H^{p(\cdot)}(\mathbb{R}^{n})$ as the set of all $f \in S^{\prime}
(\mathbb{R}^{n})$ for which $\mathcal{M}_{N}f \in L^{p(\cdot)}(\mathbb{R}^{n}).$ In this case we write $\left\Vert f\right\Vert _{H^{p(\cdot)}
}=\left\Vert \mathcal{M}_{N}f\right\Vert _{p(\cdot)}$.
\end{definition}

Let $\phi \in \mathcal{S}(\mathbb{R}^{n})$ be a function such that $\int \phi(x) dx \neq 0$. For $f \in \mathcal{S}'(\mathbb{R}^{n})$, we define the maximal function $M_{\phi}f$ by
$$M_{\phi}f(x)= \sup_{t>0} \left\vert \left(  t^{-n}\phi(t^{-1} \, \cdot)\ast f\right)  \left(  x\right)\right\vert.$$
Theorem 3.1 in \cite{Uribe} asserts that the quantities $\| M_{\phi}f \|_{p(\cdot)}$ and $\left\Vert \mathcal{M}_{N}f\right\Vert _{p(\cdot)}$ are comparable, with bounds that depend only on $p(\cdot)$ and $n$ and not on $f$, if $N > \frac{n}{p_0} + n + 1$.

Now, we give the definition of atoms.

\begin{definition}
Given $p(\cdot) \in M\mathcal{P}_0$, and $1 < q \leq \infty$, a function $a(\cdot)$ is a $(p(\cdot), q)$-atom if there exists a ball $B=B(x_0, r)$ such that\newline
$a_{1})$ $\textit{supp}\left( a\right) \subset B,$\newline
$a_{2})$ $\left\Vert a\right\Vert _{q}\leq \left\vert B\right\vert
^{\frac{1}{q}}\left\Vert \chi _{B}\right\Vert _{p(\cdot)}^{-1},$\newline
$a_{3})$ $\int a(x)x^{\alpha }dx=0$ for all $\left\vert \alpha \right\vert
\leq \lfloor n (\frac{1}{p_0}-1) \rfloor.$
\end{definition}

\begin{remark}
Let $a(\cdot)$ be a $(p(\cdot), q)$-atom and $1< s < q$, then H\"older's inequality implies $\| a \|_{s} \leq \frac{|B|^{1/s}}{\| \chi_B \|_{p(\cdot)}}$.
\end{remark}

Given $1 < q < \infty$, let $H^{p(\cdot),q}_{fin}(\mathbb{R}^{n})$ be the subspace of $H^{p(\cdot)}(\mathbb{R}^{n})$ consisting of all $f$ that have decompositions as finite sums of $(p(\cdot), q)$-atoms. By Theorem 7.1 in \cite{Uribe}, if $q$ is sufficiently large, $H^{p(\cdot),q}_{fin}(\mathbb{R}^{n})$ is dense in $H^{p(\cdot)}(\mathbb{R}^{n})$.

For $f \in H^{p(\cdot),q}_{fin}(\mathbb{R}^{n})$, define
\[
\| f \|_{H^{p(\cdot),q}_{fin}} = \inf \left\{ \left\| \sum_{j=1}^{k} \lambda_j \frac{\chi_{B_j}}{\| \chi_{B_j}\|_{p(\cdot)}} \right\|_{p(\cdot)} : f = \sum_{j=1}^{k} \lambda_j a_j \right\},
\]
where the infimum is taken over all finite decompositions of $f$ using $(p(\cdot), q)$-atoms.
Theorem 7.8 in \cite{Uribe} asserts that $\| f \|_{H^{p(\cdot),q}_{fin}} \simeq \| f \|_{H^{p(\cdot)}}$ for all $f \in H^{p(\cdot),q}_{fin}(\mathbb{R}^{n})$.

\qquad

A weight is a non-negative locally integrable function on $\mathbb{R}^{n}$ that takes values in $(0, \infty)$ almost everywhere, i.e. : the weights are allowed
to be zero or infinity only on a set of Lebesgue measure zero.

Given a weight $w$ and a measurable set $E$, we use the notation $w(E) = \int_{E} w(x) dx$. Let $f$ be a locally integrable function on $\mathbb{R}^{n}$. The function
$$\widetilde{M}(f)(x) = \sup_{\delta > 0} \frac{1}{|B(x, \delta)|} \int_{B(x, \delta)} |f(y)| dy,$$
is called the centered Hardy-Littlewood maximal function of $f$.
It is easy to check that
\begin{equation}
2^{-n}M(f)(x) \leq \widetilde{M}(f)(x) \leq M(f)(x), \,\,\,\,\, for \,\, all \,\, x \in \mathbb{R}^{n}. \label{ineqmax}
\end{equation}
We say that a weight $w \in \mathcal{A}_1$ if there exists $C >  0$ such that
\begin{equation}
M(w)(x) \leq C w(x), \,\,\,\,\, a.e. \, x \in \mathbb{R}^{n}, \label{A1cond}
\end{equation}
the best possible constant is denoted by $[w]_{\mathcal{A}_1}$. Equivalently, a weight $w \in \mathcal{A}_1$ if there exists $C >  0$ such that for every ball $B$
\begin{equation}
\frac{1}{|B|} \int_{B} w(x) dx \leq C \, ess\inf_{x \in B} w(x). \label{A1condequiv}
\end{equation}
\begin{remark}\footnote{$\mathcal{O}(n)=\{ A \in GL_n(\mathbb{R}) : A^{t}= A^{-1} \}$}
The orthogonal group $\mathcal{O}(n)$ induces an action on functions by $f_{A}(x) = f(A^{-1}x)$, where $A \in \mathcal{O}(n)$.
Since $\widetilde{M}(w_{A})(x) = [\widetilde{M}(w)]_{A}(x)$ for all $x \in \mathbb{R}^{n}$, and taking account (\ref{ineqmax}) and (\ref{A1cond}), it follows that $w \in \mathcal{A}_1$ if and only if $w_{A} \in \mathcal{A}_1$ for all $A \in \mathcal{O}(n)$. Therefore, the space of weights $\mathcal{A}_1$ is preserved by the action of $\mathcal{O}(n)$.
\end{remark}
\begin{remark} If $w \in \mathcal{A}_1$ and $0 < r < 1$, then by H\"older inequality we have that $w^{r} \in \mathcal{A}_1$.
\end{remark}
For $1 < p < \infty$, we say that a weight $w \in \mathcal{A}_p$ if there exists $C> 0$ such that for every ball $B$
$$\left( \frac{1}{|B|} \int_{B} w(x) dx \right) \left( \frac{1}{|B|} \int_{B} [w(x)]^{-\frac{1}{p-1}} dx \right)^{p-1} \leq C.$$
It is well known that $\mathcal{A}_{p_1} \subset \mathcal{A}_{p_2}$ for all $1 \leq p_1 < p_2 < \infty$.

Given $1 < p \leq q < \infty$, we say that a weight $w \in \mathcal{A}_{p,q}$ if there exists $C> 0$ such that for every ball $B$
$$\left( \frac{1}{|B|} \int_{B} [w(x)]^{q} dx \right)^{1/q} \left( \frac{1}{|B|} \int_{B} [w(x)]^{p'} dx \right)^{1/p'} \leq C < \infty.$$
For $p=1$, we say that a weight $w \in \mathcal{A}_{1,q}$ if there exists $C> 0$ such that for every ball $B$
$$\left( \frac{1}{|B|} \int_{B} [w(x)]^{q} dx \right)^{1/q} \leq C \, ess\inf_{x \in B} w(x).$$
When $p=q$, this definition is equivalent to $w^{p} \in \mathcal{A}_{p}$.
\begin{remark}
From the inequality in (\ref{A1condequiv}) it follows that if a weight $w \in \mathcal{A}_1$, then $0 < ess\inf_{x \in B} w(x) < \infty$ for each ball $B$. Thus $w \in \mathcal{A}_1$ implies that $w^{\frac{1}{q}} \in \mathcal{A}_{p,q}$, for each $1 \leq p \leq q < \infty$.
\end{remark}

A weight satisfies the reverse H\"older inequality with exponent $s > 1$, denoted by $w \in RH_{s}$, if there exists $C> 0$ such that for every ball $B$,
$$\left(\frac{1}{|B|} \int_{B} [w(x)]^{s} dx \right)^{\frac{1}{s}} \leq C \frac{1}{|B|} \int_{B} w(x) dx;$$
the best possible constant is denoted by $[w]_{RH_s}$. We observe that if $w \in RH_s$, then by H\"older's inequality, $w \in RH_t$ for all $1 < t < s$, and
$[w]_{RH_t} \leq [w]_{RH_s}$.

\begin{lemma} Given $w \in \mathcal{A}_1$, then $w \in RH_s$, where $s= 1 + (2^{n+1} [w]_{\mathcal{A}_1})^{-1}$.
\end{lemma}

This result was proved for cubes in \cite{lerner}. However, since $w \in \mathcal{A}_1$ is doubling, the Lemma holds for balls with same exponent.

Given $0<\alpha <n,$ we define the fractional maximal operator $M_{\alpha }$
by
\[
M_{\alpha }f(x)=\sup\limits_{B}\frac{1}{\left\vert B\right\vert ^{1-\frac{
\alpha }{n}}}\int\limits_{B}\left\vert f(y)\right\vert dy,
\]
where $f$ is a locally integrable function and the supremum is taken over
all the balls $B$ which contain $x.$ In the case $\alpha =0,$ the fractional
maximal operator reduces to the Hardy-Littlewood maximal operator.

The fractional maximal operator satisfies the following weighted vector-valued inequality.

\begin{lemma} Given $0< \alpha < n$, let $p$ and $q$ such that $1 < p \leq q < \infty$ and $\frac{1}{q} = \frac{1}{p} - \frac{\alpha}{n}$. Then for all $w \in \mathcal{A}_1$ and all $1 < \theta < \infty$, we have
\[
\left\| \left\{\sum_{j=1}^{\infty} [M_{\alpha}(f_j)]^{\theta} \right\}^{\frac{1}{\theta}} \right\|_{L^{q}(w)} \leq C \left\| \left\{\sum_{j=1}^{\infty} |f_j|^{\theta} \right\}^{\frac{1}{\theta}} \right\|_{L^{p}(w^{p/q})},
\]
for all sequences of functions $\{ f_j\}_{j=1}^{\infty} \subset L^{p}(w^{p/q})$.
\end{lemma}
\begin{proof} Given $w \in \mathcal{A}_1$, from Remark 2, we have that $w^{1/q} \in \mathcal{A}_{p,q}$. So, Lemma follows from Theorem 3.23 in \cite{uribe2}.
\end{proof}

We conclude these preliminaries with the definition of the weighted Hardy spaces.
Given a weight $w \in \mathcal{A}_1$ and $p_0 > 0$, the weighted Hardy space $H^{p_0}(w)$ consists of all tempered distributions $f$ such that
\[
\| f \|_{H^{p_0}(w)} = \| M_{\phi}f \|_{L^{p_0}(w)} = \left( \int_{\mathbb{R}^{n}}  [M_{\phi}f(x)]^{p_0} w(x) dx \right)^{1/p_0} < \infty.
\]
For $N$ sufficiently large, we have  $\| M_{\phi}f \|_{L^{p_0}(w)} \simeq \| \mathcal{M}_{N}f \|_{L^{p_0}(w)}$, (see \cite{tor}).

Let $p(\cdot) \in M\mathcal{P}_0$, and $q>1$. Given $w \in \mathcal{A}_1$, define $H^{p_0, q}_{fin}(w)$ as the set of all finite sums of $(p(\cdot), q)$-atoms. If $q$ sufficiently large, then by Lemma 7.6 in \cite{Uribe}, $H^{p_0, q}_{fin}(w) \subset H^{p_0}(w)$. Moreover, by Lemma 7.3 in \cite{Uribe}, we have that $H^{p(\cdot), q}_{fin}(\mathbb{R}^{n}) = H^{p_0, q}_{fin}(w)$ as sets. (It introduce this notation involving $w$ to stress that it is a subset of $H^{p_0}(w)$).

For $f \in H^{p_0, q}_{fin}(w)$, define
\[
\| f \|_{H^{p_0,q}_{fin}(w)} = \inf \left\{ \left\| \sum_{j=1}^{k} \lambda_j^{p_0} \frac{\chi_{B_j}}{\| \chi_{B_j}\|_{p(\cdot)}^{p_0}} \right\|_{L^{1}(w)}^{1/p_0} : f = \sum_{j=1}^{k} \lambda_j a_j \right\},
\]
where the infimum is taken over all finite decompositions of $f$ using $(p(\cdot), q)$-atoms. If $w \in \mathcal{A}_1 \cap L^{\left(\frac{p(\cdot)}{p_0}\right)'}$, then Lemma 7.11 in \cite{Uribe} asserts that $\| f \|_{H^{p_0, q}_{fin}(w)} \leq C \| f \|_{H^{p_0}(w)}$ for all $f \in H^{p_0, q}_{fin}(w)$.

\section{Auxiliary Results}

The following lemmas are crucial to get the main results.

\begin{lemma} For $0 \leq \alpha < n$ and $m >1$, let $T_{\alpha, m}$ be the operator defined by (\ref{T}). If $w \in \mathcal{A}_1$, then
\begin{equation}
w\left( \{ x : |T_{\alpha, m}f (x) | \geq \lambda \} \right) \leq C \lambda^{-\frac{n}{n - \alpha}} \sum_{i=1}^{m} \left(\int_{\mathbb{R}^{n}} |f(x)| [w_{A_i^{-1}}(x)]^{\frac{n- \alpha}{n}} dx \right)^{\frac{n}{n - \alpha}}. \label{tipodebil}
\end{equation}
for all integrable function $f$ with support compact.
\end{lemma}
\begin{proof}
We study separately the cases $0 < \alpha < n$ and $\alpha = 0$. For $0 < \alpha < n$, we have that $|T_{\alpha, m} f (x)| \leq \sum_{i=1}^{m} \left( I_{\alpha}|f| \right)_{A_i}(x)$, so
$$\{ x : |T_{\alpha, m}f (x) | \geq \lambda \} \subset \bigcup_{i=1}^{m} A_i \left( \left\{ x : \left( I_{\alpha}|f| \right)(x) \geq \lambda / m \right\} \right).$$
Since $w \in \mathcal{A}_1$, from Remark 9 and Remark 11, it follows that $[w_{A_i^{-1}}]^{\frac{n - \alpha}{n}} \in \mathcal{A}_{1, \frac{n}{n - \alpha}}$ for each $i=1, 2, ..., m$. Now Theorem 5 in \cite{muck} gives (\ref{tipodebil}). \\
The proof for $\alpha= 0$ is analogous to the proof of Theorem 1 b) in \cite{riveros}.
\end{proof}

\begin{lemma}
For $0 \leq \alpha < n$ and $m >1$, let $T_{\alpha, m}$ be the operator defined by (\ref{T}) and let $a(\cdot)$ be a $(p(.), q/p_0)$-atom supported on a ball $B$.\\
\textbf{a$)$} If $0 < \alpha < n$, $w \in \mathcal{A}_1$ and $q > \frac{n p_0}{\alpha}$, then for $\frac{1}{q_0} = \frac{1}{p_0} - \frac{\alpha}{n}$
\[
\int_{\mathbb{R}^{n}} |T_{\alpha, m} a (x)|^{q_0} w(x) dx \leq C |B|^{\frac{\alpha}{n}q_0} \| \chi_{B} \|_{p(.)}^{-q_0} \sum_{i=1}^{m} w_{A_{i}^{-1}}(B),
\]
\textbf{b$)$} If $\alpha=0$ and $w \in \mathcal{A}_1 \cap RH_{\left( \frac{q}{p_0}\right)'}$, then
\[
\int_{\mathbb{R}^{n}} |T_{0, m} a (x)|^{p_0} w(x) dx \leq C \| \chi_{B} \|_{p(.)}^{-p_0} \sum_{i=1}^{m} w_{A_{i}^{-1}}(B),
\]
\end{lemma}
\begin{proof}
\textbf{a$)$} Let $B = B(x_0, r)$ be the ball which $a$ is supported, we put $B^{\ast}_i = B(A_i x_0, 2r)$ with $i=1, ..., m$. We decompose $\mathbb{R}^{n} = \bigcup_{i=1}^{m} B^{\ast}_{i} \cup R$, where $R = \mathbb{R}^{n} \setminus \left( \bigcup_{i=1}^{m} B^{\ast}_{i} \right)$ and write
\[
\int_{\mathbb{R}^{n}} |T_{\alpha, m} a (x)|^{q_0} w(x) dx = \int_{\bigcup_{i=1}^{m} B^{\ast}_{i}} |T_{\alpha, m} a (x)|^{q_0} w(x) dx + \int_{R} |T_{\alpha, m} a (x)|^{q_0} w(x) dx
\]
\[
 = I_1 + I_2.
 \]
We first estimate $I_1$,
\[
I_1 \leq \sum_{i=1}^{m} \int_{B^{\ast}_{i}} |T_{\alpha, m} a (x)|^{q_0} w(x) dx = \sum_{i=1}^{m} \int_{0}^{\infty} q_0\lambda^{q_0 -1} w(\{ x \in B^{\ast}_{i} : |T_{\alpha, m} a (x)| > \lambda \}) d\lambda
\]
for $i=1, ..., m$, we write $w_{i}(x) = [w_{A_i^{-1}}(x)]^{\frac{n- \alpha}{n}}$, Lemma 14 gives
\[
\leq C \sum_{i=1}^{m} \int_{0}^{\infty} \lambda^{q_0 -1} \min \left\{ w(B^{\ast}_{i}), \lambda^{-\frac{n}{n-\alpha}} \sum_{i=1}^{m} \|a\|_{L^{1}(w_i)}^{\frac{n}{n - \alpha}} \right\} d\lambda
\]
\[
\leq C \left(\sum_{i=1}^{m} w(B^{\ast}_i) \right) \int_{0}^{\beta} \lambda^{q_0 -1} d\lambda + C \left( \sum_{i=1}^{m} \|a\|_{L^{1}(w_i)} \right )^{\frac{n}{n - \alpha}} \int_{\beta}^{\infty} \lambda^{q_0 - 1- \frac{n}{n - \alpha}} d\lambda
\]
taking $\beta = \left( \sum_{i=1}^{m} w(B^{\ast}_i) \right)^{-\frac{n - \alpha}{n}}  \left( \sum_{i=1}^{m} \|a\|_{L^{1}(w_i)} \right )$, we get
\[
I_1 \leq C \left( \sum_{i=1}^{m} w(B^{\ast}_i) \right)^{1 - \frac{(n - \alpha)}{n} q_0}  \left( \sum_{i=1}^{m} \|a\|_{L^{1}(w_i)} \right )^{q_0}.
\]
Now we estimate $\sum_{i=1}^{m} \|a\|_{L^{1}(w_i)}$, by H\"older's inequality and since $a(.)$ is a $(p(.), q/p_0)$ atom we get that
\[
\|a\|_{L^{1}(w_i)} = \int |a(x)| [w_{A_i^{-1}}(x)]^{\frac{n- \alpha}{n}} dx
\]
\[
\leq \left( \int_{B} |a(x)|^{q/p_0} dx \right)^{p_0/q} \left( \int_{B} [w_{A_i^{-1}}(x)]^{\frac{(n- \alpha)}{n} (\frac{q}{p_0})'} dx \right)^{1/(\frac{q}{p_0})'}
\]
\[
\leq |B|^{1/(q/p_0)'} |B|^{p_0/q} \| \chi_{B} \|_{p(.)}^{-1}  \left( \frac{1}{|B|} \int_{B} [w_{A_i^{-1}}(x)]^{\frac{(n- \alpha)}{n} (\frac{q}{p_0})'} dx \right)^{1/(\frac{q}{p_0})'}
\]
the condition $q > \frac{n p_0}{\alpha}$ implies $\frac{(n- \alpha)}{n} (\frac{q}{p_0})' < 1$, then from H\"older's inequality we obtain
\[
\leq |B| \| \chi_{B} \|_{p(.)}^{-1}  \left( \frac{1}{|B|} \int_{B} w_{A_i^{-1}}(x) dx \right)^{\frac{n - \alpha}{n}} = |B|^{\frac{ \alpha}{n}}
\| \chi_{B} \|_{p(.)}^{-1} [w_{A_i^{-1}}(B)]^{\frac{(n- \alpha)}{n}},
\]
so
\[
\left( \sum_{i=1}^{m} \|a\|_{L^{1}(w_i)} \right )^{q_0} \leq |B|^{\frac{ \alpha}{n}q_0}
\| \chi_{B} \|_{p(.)}^{-q_0} \left( \sum_{i=1}^{m}[w_{A_i^{-1}}(B)]^{\frac{(n- \alpha)}{n}} \right)^{q_0}
\]
\[
\leq C |B|^{\frac{ \alpha}{n}q_0}
\| \chi_{B} \|_{p(.)}^{-q_0} \left( \sum_{i=1}^{m}[w_{A_i^{-1}}(B)] \right)^{\frac{(n- \alpha)}{n} q_0}.
\]
Since $w_{A_{i}^{-1}} \in \mathcal{A}_1$ for each $i=1, ..., m$ (see Remark 9) and the weights in $\mathcal{A}_1$ are doubling measures (see Remark 7.7 in \cite{Uribe}), it follows that
\[
I_1 \leq C |B|^{\frac{\alpha}{n}q_0} \| \chi_{B} \|_{p(.)}^{-q_0} \sum_{i=1}^{m} w_{A_{i}^{-1}}(B).
\]
To estimate $I_2$, we start with a pointwise estimate. Let $d = \lfloor n(1/p_0 - 1) \rfloor$.
We denote $k(x,y)=\left\vert x-A_{1}y\right\vert ^{-\alpha
_{1}}...\left\vert x-A_{m}y\right\vert ^{-\alpha _{m}}.$ In view of the moment condition of $a(.)$ we have
\[
T_{\alpha, m} a(x)=\int\limits_{B(x_0, r)}k(x,y)a(y)dy=\int\limits_{B(x_0, r)}\left( k(x,y)-q_{d}\left(
x,y\right) \right) a(y)dy,
\]
\newline
where $q_{d}$ is the degree $d$ Taylor polynomial of the function $
y\rightarrow k(x,y)$ expanded around $x_0$. By the standard estimate of the
remainder term of the Taylor expansion, there exists $\xi $ between $y$ and
$x_0$ such that
\[
\left\vert k(x,y)-q_{d}\left( x,y\right) \right\vert \lesssim \left\vert
y-x_0 \right\vert ^{d+1}\sum\limits_{k_{1}+...+k_{n}=d+1}\left\vert \frac{%
\partial ^{d+1}}{\partial y_{1}^{k_{1}}...\partial y_{n}^{k_{n}}}k(x,\xi
)\right\vert
\]
\[
\leq C \left\vert y-x_0 \right\vert ^{d+1}\left(
\prod\limits_{i=1}^{m}\left\vert x-A_{i}\xi \right\vert ^{-\alpha
_{i}}\right) \left( \sum\limits_{l=1}^{m}\left\vert x-A_{l}\xi \right\vert
^{-1}\right) ^{d+1}.
\]
Now, we decompose $R = \bigcup_{k=1}^{m} R_{k}$ where
$$R_{k} = \{ x \in R : |x - A_k x_0| \leq |x - A_i x_0| \,\, for \,\, all \,\, i \neq k \}.$$
If $x \in R$ then $|x - A_i x_0| \geq 2r$, since $\xi \in B$ it follows that $|A_i x_0 - A_i \xi | \leq  r \leq \frac{1}{2} |x - A_i x_0|$ so
$$|x - A_i \xi| = |x - A_i x_0 + A_i x_0 - A_i \xi| \geq |x - A_i x_0| - |A_i x_0 - A_i \xi| \geq \frac{1}{2} |x - A_i x_0|.$$
If $x \in R$, then $x \in R_{k}$ for some $k$ and since $\alpha_{1}+...+\alpha_{m} = n - \alpha$ we obtain
$$
\left\vert k(x,y)-q_{d}\left( x,y\right) \right\vert  \leq C
\left\vert y-x_0 \right\vert ^{d+1}\left( \prod\limits_{i=1}^{m}\left\vert
x-A_{i}x_0 \right\vert ^{-\alpha _{i}}\right) \left(
\sum\limits_{l=1}^{m}\left\vert x-A_{l}x_0 \right\vert ^{-1}\right) ^{d+1}
$$
$$
\leq C r^{d+1}\left\vert x-A_{k}x_0 \right\vert ^{-n+\alpha -d-1},
$$
this inequality allow us to conclude that
\begin{eqnarray*}
\left\vert T_{\alpha, m}a(x)\right\vert  \leq C\left\Vert a \right\Vert
_{1}r^{d+1}\left\vert x-A_{k}x_0\right\vert ^{-n+\alpha -d-1}\\
\leq C \left\vert B \right\vert ^{1-\frac{p_0}{q}}\left\Vert
a \right\Vert _{q/p_{0}}r^{d+1}\left\vert x-A_{k}x_0\right\vert ^{-n+\alpha -d-1},
\end{eqnarray*}
since $\|a \|_{q/p_0} \leq |B|^{p_0/q} \|\chi_{B} \|_{p(\cdot)}^{-1}$, we have
\begin{eqnarray}
\left\vert T_{\alpha, m}a(x)\right\vert  \leq C\frac{r^{n+d+1}}{\left\Vert \chi _{B}\right\Vert _{p(.)}}%
\left\vert x-A_{k}x_0 \right\vert ^{-n+\alpha -d-1} \label{Tmalpha}
\end{eqnarray}
\begin{eqnarray*}
\leq C \frac{\left( M_{\frac{\alpha n}{n+d+1}}\left( \chi _{B}\right)
(A_{k}^{-1}x)\right) ^{\frac{n+d+1}{n}}}{\left\Vert \chi _{B}\right\Vert
_{p(.)}}, \,\,\,\,\, if \,\, x \in R_{k}.
\end{eqnarray*}
This pointwise estimate gives
$$
I_2 = \int_{R} |T_{\alpha, m} a (x)|^{q_0} w(x) dx
\leq C \sum_{k=1}^{m} \int_{\mathbb{R}^{n}}
\frac{\left( M_{\frac{\alpha n}{n+d+1}}\left( \chi _{B}\right)
(A_{k}^{-1}x)\right) ^{q_0 \frac{n+d+1}{n}}}{\left\Vert \chi _{B}\right\Vert
_{p(.)}^{q_0}} w(x) dx
$$
$$
= C \| \chi_{B} \|_{p(.)}^{-q_0} \sum_{k=1}^{m} \int_{\mathbb{R}^{n}}
\left( M_{\frac{\alpha n}{n+d+1}}\left( \chi _{B}\right)
(x)\right) ^{q_0 \frac{n+d+1}{n}} w_{A_{k}^{-1}}(x) dx.
$$
Since $d = \lfloor n(1/p_0 - 1) \rfloor$, we have $q_0 \frac{n+d+1}{n} > 1$. We write $\widetilde{q} = q_0 \frac{n+d+1}{n}$ and let $\frac{1}{\widetilde{p}} =
\frac{1}{\widetilde{q}} + \frac{\alpha}{n+d+1}$, so $\frac{\widetilde{p}}{\widetilde{q}} = \frac{p_0}{q_0}$ and $(w_{A_{k}^{-1}})^{1/\widetilde{q}} \in \mathcal{A}_{\widetilde{p}, \widetilde{q}}$. From Theorem 3 in \cite{muck}, we obtain
$$
\int_{\mathbb{R}^{n}}
\left( M_{\frac{\alpha n}{n+d+1}}\left( \chi _{B}\right)(x)\right) ^{q_0 \frac{n+d+1}{n}} w_{A_{k}^{-1}}(x) dx
\leq \left( \int_{\mathbb{R}^{n}} \chi_{B}(x)  [w_{A_{k}^{-1}}(x)]^{p_0/q_0} dx  \right)^{q_0/p_0}
$$
$$
\leq |B|^{\frac{\alpha}{n} q_0} w_{A_{k}^{-1}}(B),
$$
where H\"older's inequality gives the last inequality.

\textbf{b$)$} As in \textbf{a$)$} we decompose $\mathbb{R}^{n} = \bigcup_{i=1}^{m} B^{\ast}_{i} \cup R$, where $R = \mathbb{R}^{n} \setminus \left( \bigcup_{i=1}^{m} B^{\ast}_{i} \right)$, but now we write
\[
\int_{\mathbb{R}^{n}} |T_{0, m} a (x)|^{p_0} w(x) dx = \int_{\bigcup_{i=1}^{m} B^{\ast}_{i}} |T_{0, m} a (x)|^{p_0} w(x) dx + \int_{R} |T_{0, m} a (x)|^{p_0} w(x) dx
\]
\[
 = I_1 + I_2.
\]
A similar computation to done in \textbf{a$)$} allows us to obtain
\[
I_1 \leq C \left( \sum_{i=1}^{m} w(B^{\ast}_i) \right)^{1 - p_0}  \left( \sum_{i=1}^{m} \int_B |a(x)| w_{A_{i}^{-1}}(x) dx \right )^{p_0}.
\]
To estimate the last integral we use that $a(\cdot)$ is an $(p(\cdot), q/p_0)$ atom and $w \in RH_{\left( \frac{q}{p_0}\right)'}$, so
\[
\int_B |a(x)| w_{A_{i}^{-1}}(x) dx \leq |B|^{1/(q/p_0)'} |B|^{p_0/q} \| \chi_{B} \|_{p(.)}^{-1}  \left( \frac{1}{|B|} \int_{B} [w_{A_i^{-1}}(x)]^{ (\frac{q}{p_0})'} dx \right)^{1/(\frac{q}{p_0})'}
\]
\[
\leq C \| \chi_{B} \|_{p(.)}^{-1} w_{A_i^{-1}}(B).
\]
Therefore
\[
I_1 \leq C \| \chi_{B} \|_{p(.)}^{-p_0} \sum_{i=1}^{m} w_{A_{i}^{-1}}(B).
\]
To estimate $I_2$, following a similar argument to that used in \textbf{a$)$}, we get
\[
I_2 = \int_{R} |T_{0, m}a(x)|^{p_0} w(x) dx \leq C \| \chi_B \|_{p(\cdot)}^{-p_0}\sum_{i=1}^{m} \int_{\mathbb{R}^{n}} [M(\chi_B)(x)]^{p_0 \frac{n+d+1}{n}} w_{A_{i}^{-1}}(x) dx,
\]
where $d= \lfloor n(\frac{1}{p_0}-1) \rfloor$, so $p_0 \frac{n+d+1}{n} > 1$. Finally, since $w_{A_{i}^{-1}} \in \mathcal{A}_1 \subset \mathcal{A}_{p_0 \frac{n+d+1}{n}}$ for each $i=1, ..., m$, from Theorem 9 in \cite{Muck}, it follows that
\[
I_2 \leq C \| \chi_{B} \|_{p(.)}^{-p_0} \sum_{i=1}^{m} w_{A_{i}^{-1}}(B).
\]
The proof is therefore concluded.
\end{proof}
\begin{proposition}
For $0 \leq \alpha < n$ and $m >1$, let $T_{\alpha, m}$ be the operator defined by (\ref{T}). Let $p(\cdot) \in M\mathcal{P}_0$, with $0 < p_0< \frac{n}{n+ \alpha}$, such that $p(A_ix)=p(x)$ for all $i=1, ..., m$. If $w \in \mathcal{A}_1 \cap L^{\frac{p_0}{q_0}(\frac{p(\cdot)}{p_0})'}(\mathbb{R}^{n})$ and $0 < \alpha <n$ or $w \in \mathcal{A}_1 \cap L^{(\frac{p(\cdot)}{p_0})'}(\mathbb{R}^{n}) \cap RH_{(q/p_0)'}$ and $\alpha=0$, then for $\frac{1}{q_0} = \frac{1}{p_0} - \frac{\alpha}{n}$ we have
\[
\| T_{\alpha, m} f \|_{L^{q_0}(w)} \leq C \sum_{i=1}^{m} \| f \|_{H^{p_0}\left( [w_{A_i^{-1}}]^{\frac{p_{0}}{q_{0}}} \right)},
\]
for all $f \in H^{p_0, q/p_0}_{fin}(w)$, where $q$ is sufficiently large.
\end{proposition}
\begin{proof}
Given $f \in H^{p_0, q/p_0}_{fin}(w)$ we have that $f = \sum_{j=1}^{k} \lambda_j a_j$, where $a_j$ is a $(p(\cdot), q/p_0)$-atom supported on a ball $B_j$. Since $0 < q_0 < 1$, Lemma 15 gives
$$
\| T_{\alpha, m} f \|_{L^{q_0}(w)}^{q_0} = \int_{\mathbb{R}^{n}} |T_{\alpha, m}f(x)|^{q_0} w(x) dx \leq \sum_{j=1}^{k} \lambda_{j}^{q_0} \int_{\mathbb{R}^{n}} |T_{\alpha, m}a_j(x)|^{q_0} w(x) dx
$$
$$
\leq C \sum_{i=1}^{m} \sum_{j=1}^{k} \lambda_{j}^{q_0}  |B_j|^{\frac{\alpha}{n}q_0} \| \chi_{B_j} \|_{p(.)}^{-q_0} w_{A_{i}^{-1}}(B_j)
$$
$$
= C \sum_{i=1}^{m} \int_{\mathbb{R}^{n}} \left(\sum_{j=1}^{k} \lambda_{j}^{q_0}  |B_j|^{\frac{\alpha}{n}q_0} \| \chi_{B_j} \|_{p(.)}^{-q_0} \chi_{B_j}(x) \right) w_{A_{i}^{-1}}(x) dx
$$
the embedding $l^{p_0} \hookrightarrow l^{q_0}$ gives
\begin{equation}
\leq C \sum_{i=1}^{m} \int_{\mathbb{R}^{n}} \left\{\sum_{j=1}^{k} \left(\frac{\lambda_{j}  |B_j|^{\frac{\alpha}{n}} \chi_{B_j}(x)}{\| \chi_{B_j} \|_{p(.)}}\right)^{p_0} \right\}^{\frac{q_0}{p_0}} w_{A_{i}^{-1}}(x) dx \label{desig}
\end{equation}
it is clear that if $\alpha=0$, then the proposition follows from Lemma 7.11 in \cite{Uribe}, since $w_{A^{-1}_i} \in \mathcal{A}_1 \cap L^{(\frac{p(\cdot)}{p_0})'}(\mathbb{R}^{n})$ and $H^{p_0, q/p_0}_{fin}(w) = H^{p_0, q/p_0}_{fin}(w_{A_i^{-1}})$ as sets. For the case $0 < \alpha < n$, a computation gives $|B_j|^{\frac{\alpha}{n}} \chi_{B_j}(x) \leq \left(M_{\frac{\alpha p_0}{2}} (\chi_{B_j})(x)\right)^{\frac{2}{p_0}}$, so (\ref{desig})
$$
\leq C \sum_{i=1}^{m} \int_{\mathbb{R}^{n}} \left\{\sum_{j=1}^{k} \left( \frac{\lambda_{j}  \left(M_{\frac{\alpha p_0}{2}} (\chi_{B_j})(x)\right)^{\frac{2}{p_0}}}{\| \chi_{B_j} \|_{p(.)}}\right)^{p_0} \right\}^{\frac{q_0}{p_0}} w_{A_{i}^{-1}}(x) dx
$$
$$
= C \sum_{i=1}^{m} \left\| \left\{\sum_{j=1}^{k} \frac{\lambda_{j}^{p_0}  \left(M_{\frac{\alpha p_0}{2}} (\chi_{B_j})(.)\right)^{2}}{\| \chi_{B_j} \|_{p(.)}^{p_0}} \right\}^{\frac{1}{2}} \right\|^{\frac{2q_0}{p_0}}_{L^{\frac{2q_0}{p_0}}(w_{A_{i}^{-1}})}
$$
because $\frac{p_0}{2q_0} = \frac{1}{2} - \frac{\alpha p_0}{2n}$ and $[w_{A_i^{-1}}]^{\frac{p_0}{2q_0}} \in \mathcal{A}_{2, \frac{2q_0}{p_0}}$, by Lemma 13 we have
$$
\leq C \sum_{i=1}^{m} \left\| \left\{\sum_{j=1}^{k} \frac{\lambda_{j}^{p_0}  \chi_{B_j}(.)}{\| \chi_{B_j} \|_{p(.)}^{p_0}} \right\}^{\frac{1}{2}} \right\|^{\frac{2q_0}{p_0}}_{L^{2}\left([w_{A_{i}^{-1}}]^{\frac{p_0}{q_0}} \right)}
$$
$$
=C \sum_{i=1}^{m} \left\| \sum_{j=1}^{k} \frac{\lambda_{j}^{p_0}  \chi_{B_j}(.)}{\| \chi_{B_j} \|_{p(.)}^{p_0}}  \right\|^{\frac{q_0}{p_0}}_{L^{1}\left([w_{A_{i}^{-1}}]^{\frac{p_0}{q_0}} \right)}.
$$
Since, for each $i=1, ..., m$, $[w_{A_i^{-1}}]^{\frac{p_0}{q_0}} \in \mathcal{A}_1 \cap L^{(\frac{p(\cdot)}{p_0})'}(\mathbb{R}^{n})$ (see Lemma 1, Remark 9 and Remark 10), and $H^{p_0, q/p_0}_{fin}(w) = H^{p_0, q/p_0}_{fin}\left([w_{A_i^{-1}}]^{\frac{p_0}{q_0}} \right)$ as sets,  by Lemma 7.11 in \cite{Uribe}, we can take the infimum over all such decompositions to get
\[
\| T_{\alpha, m} f \|_{L^{q_0}(w)} \leq C \sum_{i=1}^{m} \| f \|_{H^{p_0}\left( [w_{A_i^{-1}}]^{\frac{p_{0}}{q_{0}}} \right)},
\]
for all $f \in H^{p_0, q/p_0}_{fin}(w)$.
\end{proof}

For $0 < \alpha < n$, let $I_{\alpha}$ be the Riesz potential given by
\begin{eqnarray}
I_{\alpha}f(x) = \int_{\mathbb{R}^{n}} \frac{f(y)}{|x-y|^{n- \alpha}} dy, \label{riesz potential}
\end{eqnarray}
where $f \in L^{s}(\mathbb{R}^{n})$, and $1 \leq s < \frac{n}{\alpha}$. 

We introduce the following discrete maximal, given $\phi \in \mathcal{S}(\mathbb{R}^{n})$ and $f \in \mathcal{S}'(\mathbb{R}^{n})$, we define
$$M^{d}_{\phi}f(x) = \sup_{j \in \mathbb{Z}} \left| (\phi^{j} \ast f) (x) \right|,$$
where $\phi^{j}(x) = 2^{jn} \phi(2^{j}x).$ From Lemma 3.2 and Proof of Theorem 3.3 in \cite{nakai}, it follows that for all $f \in  \mathcal{S}'(\mathbb{R}^{n})$ and all $0 < \theta < 1$
\begin{equation}
\mathcal{M}_N f(x) \leq C \left[ M \left(\left(M^{d}_{\phi}f \right)^{\theta}\right)(x)\right]^{\frac{1}{\theta}}, \,\,\, for \,\, all \,\, x \in \mathbb{R}^{n}, \label{maxdiadica}
\end{equation}
if $N$ is sufficiently large. This inequality gives the following
\begin{lemma}
If $w \in \mathcal{A}_1$ and $0 < q_0 < 1$, then $\| f \|_{H^{q_0}(w)} \leq C \| M^{d}_{\phi}f \|_{L^{q_0}(w)}$.
\end{lemma}
\begin{proof} Let $0 < \theta < q_0$. Since $\mathcal{A}_1 \subset \mathcal{A}_{\frac{q_0}{\theta}}$, then the lemma follows from Theorem 9 in \cite{Muck}.
\end{proof}

\begin{proposition} Let $0 < \alpha < n$. If $I_{\alpha}$ is the Riesz potential defined in (\ref{riesz potential}) and $a(\cdot)$ is a $(p(\cdot),q/p_0)$-atom, $ \frac{q}{p_0} >\frac{n}{\alpha}$, such that $\int x^{\beta} a(x) dx = 0$ for all $|\beta| \leq 2 \lfloor n (\frac{1}{q_0} -1 )\rfloor +3+ \lfloor \alpha \rfloor+n $, where $\frac{1}{q_0}=\frac{1}{p_0} - \frac{\alpha}{n}$, then
\begin{equation}
M_{\phi}^{d}(I_{\alpha}a)(x) \leq C |B|^{\frac{\alpha}{n}} \| \chi_{B}\|_{p(\cdot)}^{-1}\left[M(\chi_B)(x) \right]^{\frac{n+k+1}{n}}, \,\,\, if \,\, x \in \mathbb{R}^{n} \setminus B(x_0, 2r), \label{estimate0}
\end{equation}
where $B=B(x_0, r)$ is the ball which $a(\cdot)$ is supported and $k= \lfloor n (\frac{1}{q_0} -1 )\rfloor$.
\end{proposition}
\begin{proof} We observe that $2 \lfloor n (\frac{1}{q_0} -1 )\rfloor +3+ \lfloor \alpha \rfloor+n > \lfloor n(\frac{1}{p_0}-1) \rfloor$, thus $a(\cdot)$ is an atom with additional vanishing moments. 

The same argument utilized to obtain the pointwise estimate that appears in (\ref{Tmalpha}) works if we consider the operator $I_{\alpha}$ instead $T_{m, \alpha}$, so
\begin{equation}
|I_{\alpha}a(x)| \leq C_{|\beta|}\frac{r^{n+|\beta|+1}}{\| \chi_{B} \|_{p(\cdot)}} |x-x_0|^{-n +\alpha - |\beta| +1 }, \label{I alpha estimate}
\end{equation}
for all  $x \in \mathbb{R}^{n}\setminus B(x_0, 2r),$ and all $0 \leq |\beta| \leq 2 \lfloor n (\frac{1}{q_0} -1 )\rfloor +3+ \lfloor \alpha \rfloor+n $. Taking
$|\beta| = 2 \lfloor n (\frac{1}{q_0} -1 )\rfloor +3+ \lfloor \alpha \rfloor+n$ in (\ref{I alpha estimate}), a simple computation gives
\begin{equation}
|I_{\alpha}a(x)| \leq C \frac{r^{\alpha}}{\| \chi_B \|_{p(\cdot)}} \left( 1 + \frac{|x-x_0|}{r} \right)^{-2n-2k-3}, \, \forall x \in \mathbb{R}^{n} \setminus B(x_0, 2r) \label{cond1}
\end{equation}
where $k= \lfloor n (\frac{1}{q_0} -1 )\rfloor$.

Let $1 < s < n/\alpha$, from the $L^{s}-L^{\frac{sn}{n - s\alpha}}$ boundedness of $I_{\alpha}$ and Remark 8, we obtain
\begin{equation}
\|I_{\alpha}a\|_{L^{{\frac{sn}{n - s\alpha}}}(B(x_0,2r))} \leq C \| a \|_{s} \leq C \frac{|B|^{\frac{n - s\alpha}{sn}} |B|^{\frac{\alpha}{n}}}{\| \chi_B \|_{p(\cdot)}}. \label{cond2}
\end{equation}
Taibleson and Weiss in \cite{T-W} proved that
\begin{equation}
\int_{\mathbb{R}^{n}} x^{\beta} I_{\alpha}a(x) dx =0, \label{cond3}
\end{equation}
for $0 \leq | \beta | \leq \lfloor n (\frac{1}{q_0} - 1) \rfloor$.

Finally, we observe that the argument utilized in Proof of Theorem 5.2 in \cite{nakai} works in this setting, but considering now the estimates (\ref{cond1}), (\ref{cond2}) and the moment condition (\ref{cond3}). Therefore we get (\ref{estimate0}).
\end{proof}

\begin{remark} If $a(\cdot)$ is a $(p(\cdot),q/p_0)$-atom such that $\int x^{\beta} a(x) dx = 0$ for all $|\beta| \leq 2 \lfloor n (\frac{1}{q_0} -1 )\rfloor +3+ \lfloor \alpha \rfloor+n $, where $\frac{1}{q_0}=\frac{1}{p_0} - \frac{\alpha}{n}$, then from the inequality in (\ref{I alpha estimate}), it follows that 
\[
|I_{\alpha}a(x)| \leq C \frac{|B|^{\frac{\alpha}{n}}}{\| \chi_{B} \|_{p(\cdot)}} [M(\chi_B)(x)]^{\frac{n+k+1}{n}}, 
\]
for all  $x \in \mathbb{R}^{n}\setminus B(x_0, 2r),$ and $k= \lfloor n (\frac{1}{q_0} -1 )\rfloor$.
\end{remark}

\begin{proposition}
For $0 < \alpha < n$, let $I_{\alpha}$ be the Riesz potential given by (\ref{riesz potential}). Let $p(\cdot) \in M\mathcal{P}_0$, with $0 < p_0< \frac{n}{n+ \alpha}$. If $w \in \mathcal{A}_1 \cap L^{\frac{p_0}{q_0}(\frac{p(\cdot)}{p_0})'}(\mathbb{R}^{n})  \cap RH_{(q/p_0)'}$, then
\[
\| I_{\alpha} f \|_{H^{q_0}(w)} \leq C \| f \|_{H^{p_0}\left( w^{p_{0}/q_{0}} \right)}, \,\,\, for \,\, all \,\, f \in H^{p_0, q/p_0}_{fin}(w),
\]
where $q > \max\{1, p_{+}, p_{0}(1 + 2^{n+3}(\| M \|_{(p(\cdot)/p_0)'}+ \| M \|_{(q(\cdot)/q_0)'})), \frac{p_0 n}{\alpha}\}$.
\end{proposition}
\begin{proof}
We recall that in the decomposition atomic, we can always choose atoms with additional vanishing moments. This is, if $l$ is any fixed integer with $l > \lfloor n (\frac{1}{p_0} -1) \rfloor$, then in the definition of the space $H^{p(\cdot), q}_{fin}(\mathbb{R}^{n})$ we can assume that all moments up to order $l$ of our atoms are zero. Thus, given $f \in H^{p(\cdot), q/p_0}_{fin}(\mathbb{R}^{n})$, we have that $f = \sum_{j=1}^{k} \lambda_j a_j$, where $a_j$ are atoms with moment condition up to order $2 \lfloor n (\frac{1}{q_0} -1 )\rfloor +3+ \lfloor \alpha \rfloor+n$.

By Lemma 17 and since $0 < q_0 < 1$ we have
$$\int_{\mathbb{R}^{n}} (\mathcal{M}_N(I_{\alpha} f)(x))^{q_0} w(x) dx
\leq C \int_{\mathbb{R}^{n}} (M^{d}_{\phi}(I_{\alpha} f)(x))^{q_0} w(x) dx$$
$$\leq C \sum_{j=1}^{k} \lambda^{q_0}_{j} \int_{\mathbb{R}^{n}} (M^{d}_{\phi}(I_{\alpha} a_j)(x))^{q_0} w(x) dx.$$
Thus, we estimate the last integral for an arbitrary atom $a(\cdot)$ supported on a ball $B=B(x_0,r)$. 
$$\int_{\mathbb{R}^{n}} (M^{d}_{\phi}(I_{\alpha} a)(x))^{q_0} w(x) dx$$
$$= \int_{B(x_0, 2r)} (M^{d}_{\phi}(I_{\alpha} a)(x))^{q_0} w(x) dx +
\int_{(B(x_0, 2r))^{c}} (M^{d}_{\phi}(I_{\alpha} a)(x))^{q_0} w(x) dx = J_1 + J_2.$$
We first estimate $J_1$, for them we use the fact that $M^{d}_{\phi}(I_{\alpha} a)(x) \leq M(I_{\alpha} a)(x)$, for all $x \in \mathbb{R}^{n}$. We  have that $w \in \mathcal{A}_1$, and since the Hardy-Littlewood maximal operator satisfies Kolmogorov's inequality (see \cite{duoan}), we get
$$J_1 \leq C w(B(x_0, 2r))^{1-q_0} \left( \int_{\mathbb{R}^{n}}|I_{\alpha} a(x)| w(x) dx \right)^{q_0}.$$
To get the desired estimate for $J_1$, it is will suffice to show that
$$L=\int_{\mathbb{R}^{n}}|I_{\alpha} a(x)| w(x) dx \leq C \frac{|B|^{\frac{\alpha}{n}}}{\| \chi_B\|_{p(\cdot)}} w(B(x_0, 2r)).$$
To prove this, we split the integral
$$L= \int_{B(x_0, 2r)}|I_{\alpha} a(x)| w(x) dx + \int_{B(x_0, 2r)^{c}}|I_{\alpha} a(x)| w(x) dx =L_1 + L_2.$$
To estimate $L_1$, we take $1 < s < \frac{n}{\alpha}$ such that $0 < \frac{1}{s} - \frac{\alpha}{n} < \frac{p_0}{q}$, so if $\widetilde{s}$ is defined by $\frac{1}{\widetilde{s}} =  \frac{1}{s} - \frac{\alpha}{n}$, H\"older's inequality and the $L^{s}-L^{\widetilde{s}}$ boundedness of $I_{\alpha}$ give
$$L_1 \leq \| I_{\alpha} a\|_{\widetilde{s}} \left( \int_{B(x_0, 2r)} [w(x)]^{\widetilde{s}'} dx \right)^{1/{\widetilde{s}'}} \leq C \| a\|_{s} \left( \int_{B(x_0, 2r)} [w(x)]^{\widetilde{s}'} dx \right)^{1/{\widetilde{s}'}}$$
since $1 < s < \frac{q}{p_0}$, Remark 8 gives
$$\leq C \frac{|B|^{\frac{1}{s}}}{\| \chi_B\|_{p(\cdot)}} \left( \int_{B(x_0, 2r)} [w(x)]^{\widetilde{s}'} dx \right)^{1/{\widetilde{s}'}} = C   \frac{|B|^{1+\frac{\alpha}{n}}}{\| \chi_B\|_{p(\cdot)}} \left( \frac{1}{|B|} \int_{B(x_0, 2r)} [w(x)]^{\widetilde{s}'} dx \right)^{1/{\widetilde{s}'}}$$
a computation gives $1 < \widetilde{s}' < \left( \frac{q}{p_0} \right)'$, since $w \in RH_{(q/p_0)'}$ it follows that $w \in RH_{ \widetilde{s}'}$, thus
$$L_1 \leq C \frac{|B|^{\frac{\alpha}{n}}}{\| \chi_B \|_{p(\cdot)}} w(B(x_0, 2r)).$$
To estimate $L_2$, we use Remark 19 and Theorem 9 in \cite{Muck} to obtain
$$L_2 \leq  C \frac{|B|^{\frac{\alpha}{n}}}{\| \chi_B \|_{p(\cdot)}} w(B(x_0, 2r)).$$
From the estimates of $L_1$, $L_2$ and since the weight $w$ is doubling, we have that
$$J_1 \leq C  \frac{|B|^{q_0 \frac{\alpha}{n}}}{\| \chi_B \|_{p(\cdot)}^{q_0}} w(B).$$ 
Now we estimate $J_2$. By Proposition 18 and since $w \in \mathcal{A}_1 \subset \mathcal{A}_{q_0 \frac{n+k+1}{n}}$, once again by Theorem 9 in \cite{Muck}, we obtain
$$J_2 \leq C \frac{|B|^{q_0 \frac{\alpha}{n}}}{\| \chi_B \|_{p(\cdot)}^{q_0}} \int_{\mathbb{R}^{n}} [M(\chi_B)(x)]^{q_0 \frac{n+k+1}{n}} w(x) dx$$
$$ \leq C  \frac{|B|^{q_0 \frac{\alpha}{n}}}{\| \chi_B \|_{p(\cdot)}^{q_0}} \int_{\mathbb{R}^{n}} \chi_B(x) w(x) dx = C  \frac{|B|^{q_0 \frac{\alpha}{n}}}{\| \chi_B \|_{p(\cdot)}^{q_0}}  w(B).$$ Then
$$\int_{\mathbb{R}^{n}} (M^{d}_{\phi}(I_{\alpha} a)(x))^{q_0} w(x) dx = J_1 + J_2 \leq C  \frac{|B|^{q_0 \frac{\alpha}{n}}}{\| \chi_B \|_{p(\cdot)}^{q_0}} w(B).$$ 
So
$$\| I_{\alpha} f \|_{H^{q_0}(w)}^{q_0} \leq C \int_{\mathbb{R}^{n}}  \left\{\sum_{j=1}^{k} \left(\frac{\lambda_{j}  |B_j|^{\frac{\alpha}{n}} \chi_{B_j}(x)}{\| \chi_{B_j} \|_{p(.)}}\right)^{q_0} \right\} w(x) dx
$$
the embedding $l^{p_0} \hookrightarrow l^{q_0}$ gives
\begin{equation}
\leq C  \int_{\mathbb{R}^{n}} \left\{\sum_{j=1}^{k} \left(\frac{\lambda_{j}  |B_j|^{\frac{\alpha}{n}} \chi_{B_j}(x)}{\| \chi_{B_j} \|_{p(.)}}\right)^{p_0} \right\}^{\frac{q_0}{p_0}} w(x) dx \label{desig2}
\end{equation}
a computation allows us to obtain $|B_j|^{\frac{\alpha}{n}} \chi_{B_j}(x) \leq \left(M_{\frac{\alpha p_0}{2}} (\chi_{B_j})(x)\right)^{\frac{2}{p_0}}$, so (\ref{desig2})
$$
\leq C \int_{\mathbb{R}^{n}} \left\{\sum_{j=1}^{k} \left( \frac{\lambda_{j}  \left(M_{\frac{\alpha p_0}{2}} (\chi_{B_j})(x)\right)^{\frac{2}{p_0}}}{\| \chi_{B_j} \|_{p(.)}}\right)^{p_0} \right\}^{\frac{q_0}{p_0}} w(x) dx
$$
$$
= C \left\| \left\{\sum_{j=1}^{k} \frac{\lambda_{j}^{p_0}  \left(M_{\frac{\alpha p_0}{2}} (\chi_{B_j})(.)\right)^{2}}{\| \chi_{B_j} \|_{p(.)}^{p_0}} \right\}^{\frac{1}{2}} \right\|^{\frac{2q_0}{p_0}}_{L^{\frac{2q_0}{p_0}}(w)}
$$
because $\frac{p_0}{2q_0} = \frac{1}{2} - \frac{\alpha p_0}{2n}$ and $w^{\frac{p_0}{2q_0}} \in \mathcal{A}_{2, \frac{2q_0}{p_0}}$, by Lemma 13 we have
$$
\leq C \left\| \left\{\sum_{j=1}^{k} \frac{\lambda_{j}^{p_0}  \chi_{B_j}(.)}{\| \chi_{B_j} \|_{p(.)}^{p_0}} \right\}^{\frac{1}{2}} \right\|^{\frac{2q_0}{p_0}}_{L^{2}\left(w^{{p_0}/{q_0}} \right)}
$$
$$
=C  \left\| \sum_{j=1}^{k} \frac{\lambda_{j}^{p_0}  \chi_{B_j}(.)}{\| \chi_{B_j} \|_{p(.)}^{p_0}}  \right\|^{\frac{q_0}{p_0}}_{L^{1}\left(w^{{p_0}/{q_0}} \right)}.
$$
Since $w^{\frac{p_0}{q_0}} \in \mathcal{A}_1 \cap L^{(\frac{p(\cdot)}{p_0})'}(\mathbb{R}^{n})$ (see Lemma 1 and  Remark 10), and  $H^{p_0, q/p_0}_{fin}(w^{p_0/q_0})=H^{p_0, q/p_0}_{fin}(w) = H^{p(\cdot), q/p_0}_{fin}(\mathbb{R}^{n})$ as sets,  by Lemma 7.11 in \cite{Uribe}, we can take the infimum over all such decompositions to get
\[
\| I_{\alpha} f \|_{H^{q_0}(w)} \leq C \| f \|_{H^{p_0}\left( w^{p_{0}/q_{0}} \right)},
\]
for all $f \in H^{p_0, q/p_0}_{fin}(w)$.
\end{proof}

\section{Main Results}

In the sequel, we will consider $0 \leq \alpha < n$, a measurable function $p(\cdot) : \mathbb{R}^{n} \rightarrow (0, \infty)$ such that $0 < p_0 < p_{-} \leq p_{+} < \frac{n}{\alpha}$, with $0< p_0< \frac{n}{n+ \alpha}$ and $q(\cdot)$ defined by $\frac{1}{q(\cdot)} = \frac{1}{p(\cdot)} - \frac{\alpha}{n}$. If $\left(\frac{q(\cdot)}{q_0}\right)' \in M\mathcal{P}_{0}$, where $\frac{1}{q_0} = \frac{1}{p_0} - \frac{\alpha}{n}$, then by Lemma 4 it follows that $\left(\frac{p(\cdot)}{p_0}\right)' \in M\mathcal{P}_{0}$ because $\left(\frac{p(\cdot)}{p_0}\right)' = \frac{q_0}{p_0} \left(\frac{q(\cdot)}{q_0}\right)'$.
In the definition of the space $H^{p(\cdot), q/p_0}_{fin}(\mathbb{R}^{n})$ we will assume  $q > \max\{1, p_{+}, p_{0}(1 + 2^{n+3}\| M \|_{(p(\cdot)/p_0)'})\}$ if $\alpha=0$, or $q > \max\{1, p_{+}, p_{0}(1 + 2^{n+3}(\| M \|_{(p(\cdot)/p_0)'}+ \| M \|_{(q(\cdot)/q_0)'})), \frac{p_0 n}{\alpha}\}$ if $0 < \alpha < n$.

\begin{theorem}
For $0 \leq \alpha < n$ and $m >1$, let $T_{\alpha, m}$ be the operator defined by $($\ref{T}$)$. Given a measurable function $p(\cdot) : \mathbb{R}^{n} \rightarrow (0, \infty)$ such that $0 < p_0 < p_{-} \leq p_{+} < \frac{\alpha}{n}$, with $0 < p_0 < \frac{n}{n+ \alpha}$, define the function $q(\cdot)$ by $\frac{1}{q(\cdot)} = \frac{1}{p(\cdot)} - \frac{\alpha}{n}$. If $\left( \frac{q(\cdot)}{q_0} \right)' \in M\mathcal{P}_0$, where $\frac{1}{q_0} = \frac{1}{p_0} - \frac{\alpha}{n}$ and $q(A_ix) =q(x)$ for all $x$ and all $i=1, ...,m$, then $T_{\alpha, m}$ can be extended to an $H^{p(.)}\left( \mathbb{R}^{n}\right) - L^{q(.)}\left(\mathbb{R}^{n}\right)$ bounded operator.
\end{theorem}
\begin{proof} The operator $T_{\alpha, m}$ is well defined on the elements of $H^{p(\cdot), q/p_0}_{fin}(\mathbb{R}^{n})$. So given $f \in H^{p(\cdot), q/p_0}_{fin}(\mathbb{R}^{n})$, from Lemma 2, we have that
$$\| T_{\alpha, m} f\|_{q(.)}^{q_0} = \| |T_{\alpha, m} f|^{q_0} \|_{\frac{q(.)}{q_0}} \leq C  \sup \int |T_{\alpha, m} f(x)|^{q_0} |g (x)| dx$$
where the supremum is taken over all $g \in L^{(q(.)/q_0)'}$ such that $\| g \|_{\left(q(.)/q_0\right)'} \leq 1$.
Now we utilize the Rubio de Francia iteration algorithm with respect to $L^{\left(q(.)/q_0\right)'}$. Given a function $g$, define
$$\mathcal{R}g (x) = \sum_{i=0}^{\infty} \frac{M^{i}g(x)}{2^{i} \|M\|_{\left(q(.)/q_0\right)'}^{i}},$$
where $M^{0}g = g$ and, for $i \geq 1$, $M^{i}g = M \circ \cdot \cdot \cdot \circ M g$ denotes $i$ iterates of the Hardy-Littlewood maximal operator. The function $\mathcal{R}g$ satisfies: $|g(x)| \leq \mathcal{R}g(x)$ for all $x$, $\| \mathcal{R}g \|_{(q(\cdot)/q_0)'} \leq C \|g \|_{(q(\cdot)/q_0)'}$ and $\mathcal{R}g \in \mathcal{A}_1$ (if $\alpha=0$, then it also has $\mathcal{R}g \in RH_{(q/p_0)'})$; by these properties and since $H^{p(\cdot),q/p_0}_{fin}(\mathbb{R}^{n}) = H^{p_0,q/p_0}_{fin}(\mathcal{R}g)$ as sets, Proposition 16 gives
$$\int |T_{\alpha, m} f(x)|^{q_0} |g (x)| dx \leq \int |T_{\alpha, m} f(x)|^{q_0} \mathcal{R}g(x) dx$$
$$= \| T_{\alpha, m}f \|_{L^{q_0}(\mathcal{R}g)}^{q_0}
\leq C \sum_{i=1}^{m} \|f \|_{H^{p_0}\left([(\mathcal{R}g)_{A_{i}^{-1}}]^{\frac{p_0}{q_0}} \right)}^{q_0}$$
$$=C  \sum_{i=1}^{m} \left( \int [\mathcal{M}_{N}f(x)]^{p_0} [(\mathcal{R}g)_{A_{i}^{-1}}(x)]^{\frac{p_0}{q_0}} dx \right)^{\frac{q_0}{p_0}}$$
$$\leq C \| [\mathcal{M}_{N}f]^{p_0} \|_{\frac{p(.)}{p_0}}^{\frac{q_0}{p_0}} \sum_{i=1}^{m}\left\|[(\mathcal{R}g)_{A_{i}^{-1}}]^{\frac{p_0}{q_0}} \right\|_{\left(\frac{p(.)}{p_0} \right)'}^{\frac{q_0}{p_0}}$$
a computation gives $\left(\frac{p(.)}{p_0} \right)' = \frac{q_0}{p_0} \left( \frac{q(\cdot)}{q_0} \right)'$, so
$$= C \| \mathcal{M}_{N}f \|_{p(.)}^{q_0} \sum_{i=1}^{m}\left\|(\mathcal{R}g)^{\frac{p_0}{q_0}} \right\|_{\frac{q_0}{p_0}\left( \frac{q(A_i \cdot)}{q_0} \right)'}^{\frac{q_0}{p_0}}$$
since $q(A_i x) = q(x)$ for all $x \in \mathbb{R}^{n}$ and all $i=1, ..., m$ results
$$= C \|f \|_{H^{p(.)}}^{q_0} \sum_{i=1}^{m}\left\|\mathcal{R}g \right\|_{\left(\frac{q(.)}{q_0}\right)'}$$
$$\leq  C \|f \|_{H^{p(.)}}^{q_0} \left\|g \right\|_{\left(\frac{q(.)}{q_0}\right)'}.
$$ Thus
$$\| T_{\alpha, m} f\|_{L^{q(.)}} \leq C \|f \|_{H^{p(.)}},$$
for all $f \in H^{p(.), q/p_0}_{fin}(\mathbb{R}^{n})$, so the theorem follows from the density of $H^{p(.), q/p_0}_{fin}(\mathbb{R}^{n})$ in $H^{p(\cdot)}(\mathbb{R}^{n})$.
\end{proof}

\begin{remark}
Suppose $h:\mathbb{R}\rightarrow \left( 0,\infty \right) $ that satisfies the log-H\"older continuity on $\mathbb{R}$ and
$0<h_{-}\leq h_{+}<\frac{n}{\alpha }.$ Let $p(x)=h(\left\vert x\right\vert )$
for $x\in \mathbb{R}^{n}$ and for $m>1$ let $A_{1},...,A_{m}$ be $n\times n$
orthogonal matrices such that $A_{i}-A_{j}$ is invertible for $i\neq j.$ It
is easy to check that the function $p(\cdot)$ satisfies the log-H\"older continuity on $\mathbb{R}^{n}$
and also that $0<p_{-}\leq p_{+}<\frac{n}{\alpha }$ and $p(A_{i}x)\equiv p(x),$ $1\leq i\leq m.$

Another non trivial example of exponent functions and orthogonal matrices
satisfying the hypothesis of the theorem is the following:

We consider $m=2,$
$p(.):\mathbb{R}^{n}\rightarrow \left( 0,\infty \right) $ that satisfies the
log-H\"older continuity on $\mathbb{R}^{n}$, $0<p_{-}\leq p_{+}<\frac{%
n}{\alpha }$, and then we take $p_{e}(x)=p(x)+p(-x),$ $A_{1}=I$ and $%
A_{2}=-I.$
\end{remark}

\begin{remark} Observe that Theorem 21 still holds for $m=1$ and $0 < \alpha < n$. In particular, if $A_1=I$, then the Riesz potential is bounded from
$H^{p(.)}(\mathbb{R}^{n})$ into $L^{q(.)}(\mathbb{R}^{n})$.
\end{remark}

Now we study the boundedness of Riesz potentials from $H^{p(\cdot)}(\mathbb{R}^{n})$ into $H^{q(\cdot)}(\mathbb{R}^{n})$. 
\begin{theorem} Let $0 < \alpha < n$. Given $p(\cdot) : \mathbb{R}^{n} \rightarrow (0, \infty)$ a measurable function such that $0 < p_0 < p_{-} \leq p_{+} < \frac{\alpha}{n}$, with $0 < p_0 < \frac{n}{n+ \alpha}$, define the function $q(\cdot)$ by $\frac{1}{q(\cdot)} = \frac{1}{p(\cdot)} - \frac{\alpha}{n}$. If $\left( \frac{q(\cdot)}{q_0} \right)' \in M\mathcal{P}_0$, where $\frac{1}{q_0} = \frac{1}{p_0} - \frac{\alpha}{n}$, then the Riesz potential $I_{\alpha}$ can be extended to a bounded operator from $H^{p(\cdot)}(\mathbb{R}^{n})$ into $H^{q(\cdot)}(\mathbb{R}^{n})$.
\end{theorem}
\begin{proof} The operator $I_{\alpha}$ is well defined on $H^{p(\cdot), q/p_0}_{fin}(\mathbb{R}^{n})$. So given $f \in H^{p(\cdot), q/p_0}_{fin}(\mathbb{R}^{n})$, from Lemma 2, we have that
$$\| I_{\alpha} f \|_{H^{q(\cdot)}}^{q_0} \leq C \| \mathcal{M}_N(I_{\alpha}f) \|_{L^{q(\cdot)}}^{q_0} = \| (\mathcal{M}_N(I_{\alpha}f))^{q_0} \|_{L^{\frac{q(\cdot)}{q_0}}}$$
$$\leq C \sup \int_{\mathbb{R}^{n}} (\mathcal{M}_N(I_{\alpha} f)(x))^{q_0} |g(x)| dx$$
where the supremum is taken over all $g \in L^{(q(.)/q_0)'}$ such that $\| g \|_{\left(q(.)/q_0\right)'} \leq 1$.
Now we apply the Rubio de Francia iteration algorithm with respect to $L^{(q(\cdot)/q_0)'}$. Given $g \in L^{(q(\cdot)/q_0)'}$, define $\mathcal{R}g(x) =\sum_{i=0}^{\infty} \frac{M^{i}g(x)}{2^{i} \|M\|_{\left(q(.)/q_0\right)'}^{i}}$. The function $\mathcal{R}g$ satisfies: $|g(x)| \leq \mathcal{R}g(x)$ for all $x$, $\| \mathcal{R}g \|_{(q(\cdot)/q_0)'} \leq C \|g \|_{(q(\cdot)/q_0)'}$ and $\mathcal{R}g \in \mathcal{A}_1 \cap RH_{(q/p_0)'}$ because $[\mathcal{R}g]_{\mathcal{A}_1} \leq 2 \| M\|_{(q(\cdot)/q_0)'}$ and $q >  p_{0}(1 + 2^{n+3}(\| M \|_{(p(\cdot)/p_0)'}+ \| M \|_{(q(\cdot)/q_0)'}))$; by these properties and since $H^{p(\cdot),q/p_0}_{fin}(\mathbb{R}^{n}) = H^{p_0,q/p_0}_{fin}(\mathcal{R}g)$ as sets, Proposition 20 gives
$$\int_{\mathbb{R}^{n}} (\mathcal{M}_N(I_{\alpha} f)(x))^{q_0} |g(x)| dx \leq \int_{\mathbb{R}^{n}} (\mathcal{M}_N(I_{\alpha} f)(x))^{q_0} \mathcal{R}g(x) dx$$
$$=\| I_{\alpha} f \|_{H^{q_0}(\mathcal{R}g)}^{q_0} \leq C \| f \|_{H^{p_0}\left( [\mathcal{R}g]^{p_{0}/q_{0}} \right)}^{q_0}$$
$$=C\left( \int [\mathcal{M}_{N}f(x)]^{p_0} [(\mathcal{R}g)(x)]^{\frac{p_0}{q_0}} dx \right)^{\frac{q_0}{p_0}}$$
$$\leq C \| [\mathcal{M}_{N}f]^{p_0} \|_{\frac{p(.)}{p_0}}^{\frac{q_0}{p_0}} \left\|[\mathcal{R}g]^{\frac{p_0}{q_0}} \right\|_{\left(\frac{p(.)}{p_0} \right)'}^{\frac{q_0}{p_0}}$$
a computation gives $\left(\frac{p(.)}{p_0} \right)' = \frac{q_0}{p_0} \left( \frac{q(\cdot)}{q_0} \right)'$, so
$$\leq C  \|f \|_{H^{p(.)}}^{q_0} \left\|g \right\|_{\left(\frac{q(.)}{q_0}\right)'}.$$
Therefore,
$$ \| I_{\alpha} f\|_{H^{q(.)}} \leq C \|f \|_{H^{p(.)}},$$
for all $f \in H^{p(.), q/p_0}_{fin}(\mathbb{R}^{n})$, so the theorem follows from the density of $H^{p(.), q/p_0}_{fin}(\mathbb{R}^{n})$ in $H^{p(\cdot)}(\mathbb{R}^{n})$.

\end{proof}

\end{document}